\documentclass{amsart}
\usepackage{amsmath,amsfonts,amsthm}
\usepackage{amssymb,latexsym}
\usepackage[colorlinks]{hyperref}
\usepackage{graphicx}
\usepackage[all]{xy}
\usepackage{layout}

\theoremstyle{plain}
\newtheorem{theorem}{Theorem}[section]
\newtheorem{lemma}[theorem]{Lemma}
\newtheorem{proposition}[theorem]{Proposition}
\newtheorem{corollary}[theorem]{Corollary}

\theoremstyle{definition}

\newtheorem{definition}[theorem]{Definition}

\theoremstyle{remark}

\numberwithin{equation}{section}

\begin{document}

\title{finite  approximation properties of
$C^{*}$-modules II}

\author{Massoud Amini}

\address{Department of Mathematics\\ Faculty of Mathematical Sciences\\ Tarbiat Modares University\\ Tehran 14115-134, Iran}
\email{mamini@modares.ac.ir, mamini@ipm.ir}

\address{Current Address: STEM Complex, 150 Louis-Pasteur Pvt,
	Ottawa, ON, Canada K1N 6N5}




\maketitle

\begin{abstract}
We study quasidiagonality and local reflexivity for $C^{*}$-algebras which are $C^*$-module over another $C^*$-algebra
with compatible actions. We introduce and study a notion of amenability for vector valued traces.  
\end{abstract}

\section{introduction}\label{s1}

Finite  approximation properties of $C^*$-algebras is studied in \cite{bo}. Some of these,  including important notions such as nuclearity, exactness and weak expectation property (WEP) are extended to the context of $C^*$-algebras with compatible module structure in  \cite{a2}. We continue this study here by considering other important finite  approximation properties such as quasidiagonality and local reflexivity. We also study vector valued traces and their amenability. 

A ``finite dimensional approximation'' scheme for $C^*$-morphisms is an approximately commuting diagram as follows:

\begin{center}
$\xymatrix @!0 @C=4pc @R=3pc { A \ar[rr]^{\theta} \ar[rd]^{\varphi_n} && B  \\ & \mathbb M_{k_n}(\mathbb C) \ar[ur]^{\psi_n}}$
\end{center}

\noindent where $A$ and $B$ are $C^*$-algebras and $\varphi_n$ and $\psi_n$ are contractive completely positive (c.c.p.) maps. 
The central idea of the module case, where $A$ and $B$ are also $\mathfrak A$-modules, for a $C^*$-algebra $\mathfrak A$, is to find such an approximate decomposition through the $C^*$-algebra $\mathbb M_{k_n}(\mathfrak A)$ (or through the von Neumann algebra $\mathbb M_{k_n}(\mathfrak A^{**})$). This means that  we deal with approximations through finitely generated modules (over $\mathfrak A$ or $\mathfrak A^{**}$), shortly referred to as ``finite approximation'' here.  

The paper is organized as follows: In section 2 we use the notion of retraction, already studied in the context of Hilbert $C^*$-modules \cite{lan}, to introduce the notion of vector valued amenable traces on $C^*$-modules. In section 3, we study a notion of quasidiagonality in the category of $C^*$-modules and extend Voiculescu theorem (Theorem \ref{voi}). The last section is devoted to extension of local reflexivity, Arveson's lemma (Lemma \ref{lift}) and a work of Kirchberg on $min$-continuity properties.

For the rest of this paper, we fix a $C^*$-algebra $\mathfrak A$ and let $A$ be a $C^*$-algebra and a right Banach
$\mathfrak A$-module (that is, a module with contractive right action) with compatible conditions,
\begin{equation*}
(ab)\cdot\alpha=a(b\cdot\alpha),\,\, a\cdot\alpha\beta=(a\cdot\alpha)\cdot\beta,
\end{equation*}
for each $a,b\in A $ and $\alpha, \beta\in \mathfrak A.$ In this case, we  say that $A$ is a (right) $\mathfrak A$-$C^*$-module, or simply a $C^*$-module (it is then understood that the algebra and module structures on $A$ are compatible in the above sense). A $C^*$-subalgebra which is also an $\mathfrak A$-submodule is simply called a $C^*$-submodule.

If moreover, we have the compatibility condition
$$(a^*\cdot\alpha^*)^*\cdot\beta=((a\cdot\beta)^*\cdot\alpha^*)^*,$$
for each $a\in A $ and $\alpha, \beta\in \mathfrak A,$ then if we define a left action by
$$\alpha\cdot a:=(a^*\cdot\alpha^*)^*,$$
then $A$ becomes an $\mathfrak A$-bimodule with compatibility conditions,
\begin{equation*}
\alpha\cdot (ab)=(\alpha\cdot a)b,\,\, \alpha\beta\cdot a=\alpha\cdot(\beta\cdot a),\, \, \alpha\cdot(a\cdot \beta)=(\alpha\cdot a)\cdot \beta,
\end{equation*}
for each $a,b\in A $ and $\alpha, \beta\in \mathfrak A.$ In this case, there is a canonical $*$-homomorphism from $\mathfrak A$ to the multiplier algebra $M(A)$ of $A$, sending $\alpha$ to the pair $(L_\alpha, R_\alpha)$ of left and right module multiplication map by $\alpha$. If the action is {\it non degenerate}, in the sense that, given $\alpha$, $a\cdot \alpha=0$, for each $a\in A$, implies that $\alpha=0$ (and so the same for the left action), then the above map is injective and so an isometry, and we could (and would) identify $\mathfrak A$ with a $C^*$-subalgebra of $M(A)$.

We say that a two sided action of $\mathfrak A$ on $A$ is a {\it biaction} if the right and left actions are compatible, i.e.,
$$(a\cdot\alpha)b=a(\alpha\cdot b)\ \ (\alpha\in\mathfrak A, a,b\in A).$$
When the action is non degenerate and $\mathfrak A$ acts on $A$ as a $C^*$-subalgebra of $M(A)$, we have a biaction.

In some cases we have to work with operator $\mathfrak A$-modules with no algebra structure (and in particular with certain Hilbert $\mathfrak A$-modules). If $E, F$ are operator $\mathfrak A$-modules, a module map $\phi: E\to F$ is a continuous linear map which preserves the right $\mathfrak A$-module action.

Throughout this paper, we use the notation $\mathbb B(X)$ to denote the set of bounded adjointable linear operators on an Hilbert $C^*$-module  $X$.

\section{amenable traces}

In this section, $\mathfrak A$ is a $C^*$-algebra and $A$ is a right $\mathfrak A$-$C^*$-module, with the compatibility conditions which allow one to consider $A$ as an $\mathfrak A$-bimodule. When the action is non degenerate, we identify $\mathfrak A$ with a $C^*$-subalgebra of the multiplier algebra $M(A)$ (or that of $A$, when $A$ is unital).

The representations of $\mathfrak A$-$C^*$-modules are defined on $\mathfrak A$-correspondences. An $\mathfrak A$-{\it correspondence} is a right Hilbert $\mathfrak A$-module $X$ with a left action of $\mathfrak A$ via a representation of $\mathfrak A$ into $\mathbb B(X)$. A {\it representation} of a $\mathfrak A$-bimodule $A$ in $X$ is a $*$-homomorphism from $A$ into $\mathbb B(X)$, which is also a right $\mathfrak A$-module map with respect to the canonical right $\mathfrak A$-module structure of $\mathbb B(X)$ coming from the left  $\mathfrak A$-module structure of $X$.

To define a module version of (vector valued) traces, we adapt (and slightly generalize) the notion of retraction from \cite[Chapter 5]{lan}.

\begin{definition} \label{ret}
An $\mathfrak A$-{\it retraction} is a positive right $\mathfrak A$-module map $\tau: A\to \mathfrak A$ such that

$(i)$ Im$(\tau)$ is strictly dense in $\mathfrak A$, that is, for each $\beta\in\mathfrak A$, there is a net $(a_i)\subseteq A$ such that
$$\|\tau(a_i\cdot\alpha)-\beta\alpha\|\to 0\ \ (\alpha\in\mathfrak A),$$

$(ii)$ for some bounded approximate identity $(e_i)$ of $A$, $\tau(e_i)\to p$, for some projection $p$ in $\mathfrak A$, in the strict topology.
\end{definition}

Note that since $\tau$ is positive, it is also self-adjoint (i.e., it preserves the involution). In particular, an $\mathfrak A$-retraction is automatically a bimodule map (with the left action defined as in the previous section) and the left module version of conditions above are also satisfied. This observation is  used in the proof of part $(i)$ of the next lemma.

When $A$ is unital (and $p=1$) and $\mathfrak A$ is a $C^*$-subalgebra of $A$, an $\mathfrak A$-retraction is simply a conditional expectation from $A$ onto $\mathfrak A$. Each state of the $C^*$-algebra $A$ is a $\mathbb C$-retraction. As another example,  for $\mathfrak A=pM(A)p$, where $p$ is a projection in $M(A)$,  the cut down map $a\mapsto pap$ is an $\mathfrak A$-retraction.

The first part of following result is proved by a slight modification of the argument of \cite[Lemma 5.9]{lan}. Here we sketch the proof, as it also uses the bimodule property. The second part follows from the  observation (in the introduction) about non degenerate actions and \cite[Proposition 5.10]{lan}.

\begin{lemma} \label{retr}

$(i)$ An $\mathfrak A$-retraction is a c.p. map.

$(ii)$ If the action is non degenerate, $\mathfrak A$-retractions are exactly those linear maps $\tau: A\to \mathfrak A$ which have an extension to an idempotent c.c.p. map $\tilde\tau: M(A)\to \mathfrak A$, which is strictly continuous on the unit ball of $M(A)$.
\end{lemma}
\begin{proof}
We only prove part $(i)$. Identifying $\mathbb M_n(\mathfrak A)$ with $\mathbb K(\mathfrak A^n)$ \cite[Lemma 4.1]{lan}, we only need to show that
$$\sum_{i,j=1}^{n} \alpha_i^*\tau(a_i^*a_j)\alpha_j\geq 0\ \ (\alpha_1,\cdots,\alpha_n\in\mathfrak A, a_1,\cdots,a_n\in A).$$
Since $\tau$ is a bimodule map, the left hand side is the same as $\tau(b^*b)$ for $b=\sum_i a_i\cdot\alpha_i$, and we are done.\qed
\end{proof}

Next, let us assume that the action is non degenerate and regard $\mathfrak A$ as a $C^*$-subalgebra of $M(A)$. Let $X$ be a right Hilbert $A$-module. Then $X$ is also a right Hilbert $M(A)$-module: for $x\in X$, $c\in M(A)$, and bounded approximate identity $(a_i)\subseteq A$,
$$|xa_ic-xa_jc|^2=c^*(a_i-e_j)|x|^2(a_i-a_j)c\to 0,$$
as a net in $A$. Thus we may define $xc$ as the limit of the net $(xa_ic)$ in $M(A)$. When the action is non degenerate, $X$ is also a right $\mathfrak A$-module.

Assume that  the action is non degenerate. For an $\mathfrak A$-retraction, define
$$\langle x,y\rangle_\tau:=\tau(\langle x,y\rangle_A)\in\mathfrak A\ \ (x,y\in X).$$
Consider the closed submodule $N_\tau:=\{x:  \langle x,x\rangle_\tau=0\}$ of $X$. The completion of the quotient $X/N_\tau$ is a Hilbert $\mathfrak A$-module, denoted by $L^2(X, \tau)$. There is module map: $X\to  L^2(X, \tau); x\mapsto \hat x:=x+N_\tau$, satisfying $\widehat{ x\cdot\alpha}=\hat x\cdot \alpha$, for $\alpha\in\mathfrak A$.  When $X=A$, we denote this module simply by $L^2(A, \tau)$.

For $t\in\mathbb B(X)$, since $0\leq |tx|^2\leq \|t\|^2|x|^2$ \cite[Proposition 1.2]{lan}, the map
$$\pi_\tau: X/N_\tau \to X/N_\tau; \ x+N_\tau\to tx+N_\tau,$$
is well-defined and bounded, and extends to an $*$-homomorphism
$$\pi_\tau: \mathbb B(X)\to \mathbb B\big(L^2(X, \tau)\big),$$
which is injective when $\tau$ is faithful (that is, $\tau(a)=0$ implies $a=0$, for each $a\in A^+$). We say that $X$ is $\tau$-{\it self dual}, if $X$ is a self-dual Hilbert $\mathfrak A$-module under the above $\mathfrak A$-valued inner product.
In this case, $L^2(X,  \tau)$ is also a self dual Hilbert $\mathfrak A$-module (since $X/N_\tau$ is dense in $L^2(X, \tau)$ and the inner product is continuous). In particular, if $\mathfrak A$ is a von Neumann algebra and $X$ is $\tau$-self dual, then $\mathbb B\big(L^2(X, \tau)\big)$ is a von Neumann algebra, and so is the double commutant  $\pi_\tau(A)^{''}\subseteq \mathbb B\big(L^2(X, \tau)\big)$.

When $A$ is faithfully represented in $X$, say $A\subseteq \mathbb B(X)$, we may restrict this to $A$ to get an $*$-homomorphism
$$\pi_\tau: A\to \mathbb B\big(L^2(X, \tau)\big),$$
which is essentially the same as the GNS-construction of $\tau$, when $X=A$. In this case, for each $a\in A$, the operator $\pi_\tau(a)$ is defined on the dense subset $A/N_\tau\subseteq L^2(A, \tau)$ by $\pi_\tau(a)\hat b=\widehat{ab}$, and so it is justified to call $\pi_\tau$  the left regular representation associated to the $\mathfrak A$-retraction $\tau$.

If $\mathfrak A$ is unital and $(e_i)$ is the bounded approximate identity as in part $(ii)$ of the above definition, then
$$\|\hat e_i-\hat e_j\|_2=\|\tau(e_i-e_j)^2\|^{\frac{1}{2}}\leq\|\tau(e_i-e_j)\|^{\frac{1}{2}}\to 0,$$
as $i,j\to\infty$. Thus, there is $\hat 1\in L^2(A,\tau)$ with $\hat e_i\to \hat 1$ in $L^2(A,\tau)$, and $\hat 1$ is a cyclic vector for $\pi_\tau$. When $A$ is unital, $\hat 1$ is simply the canonical image of $1\in A$, but it exists, even if $A$ is not unital.

Another way to extend the GNS-construction, is adapting the so called  Kasparov-Stinespring-Gelfand-Naimark-Segal (KSGNS) construction \cite{lan} for $\mathfrak A$-retractions. Let $Y$ be a Hilbert $\mathfrak A$-module and $\rho: A\to \mathbb B(Y)$ be a c.p. map. We say that $\rho$ is {\it strict} if the net $(\rho(e_i))$ is strictly Cauchy in $\mathbb B(Y)$, for some bounded approximate identity $(e_i)\subseteq A$ (this is automatic when $A$ is unital). Since the unit ball of $\mathbb B(Y)$ is complete in strict topology \cite{lan}, the above condition implies that $\rho(e_i)\to p$, in the strict topology, for some positive element $p$ in $\mathbb B(Y)$ (the case $p=1$ is equivalent to the condition that $\rho$ is non degenerate). Now the KSGNS-construction of a strict c.p. map $\rho: A\to \mathbb B(Y)$ gives a Hilbert $\mathfrak A$-module $Y_\rho$, an adjointable operator $v\in \mathbb B(Y,Y_\rho)$ and a $*$-homomorphism $\pi_\rho: A\to \mathbb B(Y_\rho)$ with $\rho=v^*\pi_\rho(\cdot)v$, such that $\pi_\rho(A)v(Y)$ is dense in $Y_\rho$, which is universal in the sense that, for each Hilbert $\mathfrak A$-module $Z$ and $*$-homomorphism $\pi: A\to \mathbb B(Z)$ and $w\in \mathbb B(Y,Z)$ with $\rho=w^*\pi(\cdot)w$, such that $\pi(A)w(Y)$ is dense in $Z$, there is a unitary $u\in \mathbb B(Y_\rho,Z)$ with $\pi=u\pi_\rho(\cdot)u^*$ and $w=uv$ \cite[Theorem 5.6]{lan}. Indeed, $Y_\rho=A\otimes_\rho Y$ and $\pi_\rho(a)(vy)=a\dot\otimes y$, for $a\in A, y\in Y$. When $\rho$ is a non degenerate $*$-homomorphism, $Y_\rho$ is unitarily equivalent to $Y$.

Back to the $\mathfrak A$-retraction $\tau: A\to \mathfrak A$, since $\pi_\tau: A\to  \mathbb B\big(L^2(A, \tau)\big)$ is a non degenerate $*$-homomorphism, for the $\mathfrak A$-module $Y=L^2(A, \tau)$, the KSGNS-construction $Y_{\pi_\tau}$ could be identified (via a unitary equivalence) with $Y$. Under this identification, the adjointable operator $v$ above is identified with the identity and we get the above universality property for free. Another choice for $Y$ is $Y=\mathfrak A$, which gives $Y_\tau=A\otimes_\tau \mathfrak A$, with the above universal property. In this case, $v\in \mathbb B(\mathfrak A,A\otimes_\tau \mathfrak A)$ satisfies $\pi_\tau(a)(v\alpha)=a\dot\otimes \alpha$. Again, if $\mathfrak A$ is unital, the $v(1_{\mathfrak A})$ is a cyclic vector for the representation $\pi_\tau: A\to \mathbb B(A\otimes_\tau \mathfrak A)$ (indeed, $v(1_{\mathfrak A})$ is the limit of the net $(e_i\dot\otimes 1_{\mathfrak A})$ in $A\otimes_\tau \mathfrak A$).

\begin{definition}
An $\mathfrak A$-{\it trace} is an $\mathfrak A$-retraction $\tau: A\to \mathfrak A$ satisfying
$$\tau(ab)=\tau(ba)\ \ (a,b\in A).$$
\end{definition}

In this case, one could also define the right regular representation $\pi^{op}_\tau$ of $\tau$ on $A^{op}$ by $\pi^{op}_\tau(a)\hat b=\widehat{ba}$ (extended by continuity). This is well-defined, since $\tau$ is an $\mathfrak A$-trace. The proof of the next lemma goes, almost verbatim, as in the classical case \cite[6.1.2-6.1.4]{bo}.

\begin{lemma} \label{j} Let $\tau: A\to \mathfrak A$ be an $\mathfrak A$-trace.

$(i)$ $\pi_\tau(A)^{'}\supseteq \pi_\tau^{op}(A^{op})$ in $\mathbb B\big(L^2(A, \tau)\big),$

$(ii)$ there is a conjugate $\mathfrak A$-morphism $J: L^2(A, \tau)\to L^2(A, \tau)$, with $J^2=id$, satisfying $J\hat a=\widehat{a^*}$ and $J\pi_\tau(a)J=\pi_\tau^{op}(a^*)$, for each $a\in A$, and $\langle Jz, z^{'}\rangle=\langle Jz^{'}, z\rangle$ for each $z, z^{'}\in L^2(A, \tau)$.

When $\mathfrak A$ is unital, we also have,

$(iii)$  $Jt\hat 1=t^*\hat 1$, for each $t\in \pi_\tau(A)^{'}$,

$(iv)$  $\pi_\tau(A)^{''}= \pi_\tau^{op}(A^{op})^{'}$ and $\pi_\tau(A)^{'}= \pi_\tau^{op}(A^{op})^{''}$ in $\mathbb B\big(L^2(A, \tau)\big).$
\end{lemma}

In part $(ii)$, the fact that $J$ is a conjugate $\mathfrak A$-morphism simply means that $J(\alpha\cdot \hat a)=J(\hat a)\cdot\alpha^*$, and the same for the right action.

In the next definition, we use the left module structure of $Y$ to define the canonical right module action of $\mathfrak A$ on $\mathbb B(Y)$ by $$(t\cdot\alpha)(y)=t(\alpha\cdot y)\ \ (t\in\mathbb B(Y), \alpha\in \mathfrak A, y\in Y).$$

\begin{definition} \label{am}
An $\mathfrak A$-trace $\tau: A\to \mathfrak A$ is called {\it amenable} if for every faithful representation $A\subseteq \mathbb B(Y)$ of $A$ in an $\mathfrak A$-correspondence $Y$, such that $\tau$ has an extension to a  c.c.p. right module map $\phi: \mathbb B(Y)\to \mathfrak A$ satisfying $\phi(uxu^*)=\phi(x)$, for each $x\in B(Y)$ and each unitary $u$ in $A+\mathbb C I$, where $I$ is the identity operator on $Y$.
\end{definition}

Note that if $\tau$ is only an $\mathfrak A$-retraction satisfying the above extension property, then it is automatically an $\mathfrak A$-trace. When $A$ is a unital $C^*$-algebra, faithfully and non degenerately represented in $Y$, then we only need to check that $\phi$ is stable under the unitary conjugation by unitaries in $A$. Finally, and the most important of all, note that, unlike the classical case \cite[Proposition 3.1.2]{bo}, here the existence of invariant extension is not independent of the choice of the representation, unless $A$ is represented in an  $\mathfrak A$-module with injective algebra of adjointable operators.

\begin{lemma} \label{rep} Let $A\subseteq \mathbb B(Y)$ be a faithful representation in a $\mathfrak A$-correspondence $Y$ such that $\mathbb B(Y)$ is an injective $\mathfrak A$-module. Then if $\tau: A\to \mathfrak A$  is an $\mathfrak A$-trace which enjoys the above invariant extension property  for this representation, then $\tau\circ\pi^{-1}$ has the invariant extension property for any other faithful representation  $\pi: A \to\mathbb B(X)$ of $A$ in a $\mathfrak A$-correspondence $X$, that is,  $\tau$ is an amenable $\mathfrak A$-trace.
\end{lemma}
\begin{proof}
First note that $\tau\circ\pi^{-1}$ is clearly an $\mathfrak A$-retraction. By injectivity, there is a  c.c.p. module map $\Phi: \mathbb B(X)\to\mathbb B(Y)$ which extends $\pi^{-1}$ on $\pi(A)$, and  $\pi(A)$ is in the multiplicative
domain of $\Phi$. Consider the c.c.p. map $\phi: \mathbb B(Y)\to \mathfrak A$ satisfying the conditions of Definition \ref{am}, and put $\psi=\phi\circ\Phi$. This is a c.c.p. map on $\mathbb B(X)$ which extends $\tau\circ\pi^{-1}$ on $\pi(A)$ and enjoys the invariance property on $\mathbb B(X)$.
\end{proof}

The above lemma suggests that we look for appropriate  $\mathfrak A$-correspondence $Y$ such that there is a  faithful representation $A\subseteq \mathbb B(Y)$ and $\mathbb B(Y)$ is an injective $\mathfrak A$-module. One case of special interest is $Y=K\otimes \mathfrak A$, where $K$ is an appropriate Hilbert space. We use a minimal Stinespring dilation to show that $A$ could always be faithfully represented in such an space by a module map. In the following lemma, we assume that we have a biaction (this is satisfied if the action is non degenerate and $\mathfrak A$ acts on $A$ as a $C^*$-subalgebra of $M(A)$ by multiplication). This lemma relaxes the separability condition used in \cite[section 3]{a2}, showing that separability is not needed to define the {\it min} module tensor product.

\begin{lemma} \label{rep2} Let $\mathfrak A$ be unital. There is a Hilbert space $K$ and a  $\mathfrak A$-correspondence structure on $K\otimes \mathfrak A$ such that $A$ could  be faithfully represented in $K\otimes \mathfrak A$.
\end{lemma}
\begin{proof}
Let $\mathfrak A\subseteq \mathbb B(H)$ be a faithful representation of the $C^*$-algebra $\mathfrak A$ in a Hilbert space $H$. Take a faithful c.c.p. map $\varphi: A\to \mathfrak A^{'}\subseteq \mathbb B(H)$ (such a c.c.p. map always exists, for instance take any faithful state $\phi$ of $\mathfrak A$ and put $\varphi(a)=\phi(a)1_{\mathfrak A}$). Let $(\pi,K,V)$ be the minimal Stinespring dilation of $\varphi$. Then there is a $*$-homomorphism
$$\rho:\varphi(A)^{'}\to \pi(A)^{'}\subseteq \mathbb B(K)$$ such that $\varphi(a)x=V^*\pi(a)\rho(x)V$, for $a\in A, x\in \varphi(A)^{'}$ \cite[1.5.6]{bo}. Since $\varphi=V^*\pi(\cdot)V$ is faithful, so is $\pi$ (just calculate both sides of the last equality at $aa^*$). On the other hand, $K=A\otimes_\varphi H$ and $\pi(a)(b\otimes h)=ab\otimes h$, thus
\begin{align*}
\pi(a\cdot\alpha)(b\otimes h)&=(a\cdot\alpha)b\otimes h
=a(\alpha\cdot b)\otimes h\\
&=\pi(a)\big((\alpha\cdot b)\otimes h\big)
=\big(\pi(a)\cdot\alpha\big)(b\otimes h),
\end{align*}
for $\alpha\in \mathfrak A, a,b\in A, h\in H$.

Note that since $\varphi(A)\subseteq \mathfrak A^{'}$, $\varphi(A)^{'}\supseteq \mathfrak A^{''}\supseteq \mathfrak A$, so put $$\sigma=\rho|_{\mathfrak A}: \mathfrak A\to \pi(A)^{'}\subseteq \mathbb B(K)$$
and define the left action of $\mathfrak A$ on $K$ by
$$\alpha\cdot\xi:=\sigma(\alpha)\xi\ \ (\alpha\in \mathfrak A, \xi\in K),$$
and let $A$ act on the right Hilbert $\mathfrak A$-module $K\otimes \mathfrak A$ with inner product $$\langle \xi\otimes\alpha,\eta\otimes\beta\rangle:=\langle \xi,\eta\rangle\alpha^*\beta\ \ (\alpha,\beta\in \mathfrak A, \xi,\eta\in K),$$
via $\tilde\pi: A\to \mathbb B(K\otimes\mathfrak A)$, defined by
$$\tilde\pi(a)(\xi\otimes\alpha):=\pi(a)\xi\otimes\alpha\ \ (\alpha\in \mathfrak A, a\in A, \xi\in K),$$
then $\tilde\pi$ is faithful, since $\pi$ is faithful and $\mathfrak A$ is unital. Moreover, since $\pi$ and $\sigma$ have commuting ranges,
\begin{align*}
\tilde\pi(a\cdot\beta)(\xi\otimes \alpha)&=\pi(a\cdot\beta)\xi\otimes \alpha
=(\pi(a)\cdot\beta)(\xi)\otimes \alpha\\
&=\pi(a)(\beta\cdot\xi)\otimes \alpha
=\pi(a)\sigma(\beta)\xi\otimes \alpha\\
&=\sigma(\beta)\pi(a)\xi\otimes \alpha
=\beta\cdot\big(\pi(a)\xi\otimes \alpha\big)\\
&=\beta\cdot\big(\tilde\pi(a)(\xi\otimes \alpha)\big)
=\big(\tilde\pi(a)\cdot\beta\big)(\xi\otimes \alpha),
\end{align*}
for $\alpha,\beta\in \mathfrak A, a\in A, \xi\in K$.
Therefore $\tilde\pi$ is the required faithful representation.
\end{proof}

In the category of Hilbert $\mathfrak A$-modules with bounded adjointable maps, it is known that $\mathfrak A$ is an injective object iff the multiplier algebra $M(\mathfrak A)$ is monotone complete \cite[Theorem 1.1]{f}. I am not aware of conditions for injectivity of $\mathfrak A$ in our category of $C^*$-modules with c.c.p. module maps.

\begin{lemma} \label{inj} Let $\mathfrak A$ be unital. For the $\mathfrak A$-correspondence $K\otimes \mathfrak A$ of the above lemma, if $\mathbb B(K\otimes \mathfrak A)$ is a von Neumann algebra, then if $\mathfrak A$ is an injective object in the category of $\mathfrak A$-modules with c.c.p. module maps as morphisms, then so is $\mathbb B(K\otimes \mathfrak A)$.
\end{lemma}
\begin{proof}
As in the proof of the module version of Arveson extension theorem \cite[Lemma  3.7]{a2}, it is enough to note that, for an orthonormal basis $\{\xi_i\}$ of $K$, the set $\{\xi_i\otimes 1_{\mathfrak A}\}$  is a frame for the Hilbert $\mathfrak A$-module $K\otimes \mathfrak A$. This gives a net of finite rank projections $(q_i)$ (say with rank $k(i)$) in $\mathbb B(K\otimes \mathfrak A)$, tending to the identity in $SOT$, such that $q_i\mathbb B(K\otimes \mathfrak A)q_i=\mathbb M_{k(i)}(\mathfrak A)$. The rest goes as in the proof of the classical Arveson extension theorem.  
\end{proof}

Note that,
$$\mathbb B(K\otimes \mathfrak A)=M(\mathbb K(K)\otimes \mathfrak A)\subseteq \mathbb B(K)\bar\otimes \mathfrak A^{**}.$$
Also, $\mathbb B(K\otimes \mathfrak A)$ is a von Neumann algebra when $\mathfrak A$ is so and $K\otimes \mathfrak A$ is self dual.

Let  $\tau$ be an $\mathfrak A$-trace, and consider the left and right regular representations $\pi_\tau$ and $\pi^{op}_\tau$ on $L^2(A,\tau)$. Since the ranges of these representations commute, we may consider the representation $\pi_\tau\otimes\pi_\tau^{op}: A\odot A^{op}\to \mathbb B(L^2(A,\tau))$. Composing this with the $\mathfrak A$-retraction $$\theta: x\in \mathbb B(L^2(A,\tau))\mapsto\langle x\hat1, \hat 1\rangle\in\mathfrak A$$
we get a map $\mu_\tau: A\odot A^{op}\to\mathfrak A$. Both of these maps are $max$-continuous  (by universality) on  $A\odot A^{op}$. The natural question is that when these are $min$-continuous.
The main result of this section answers this and gives  sufficient conditions for amenability of $\tau$. The proof resembles that of \cite[3.1.6]{b}, but there are lots of technicalities, due to working with vector valued maps, which should be taken care of. For $x\in \mathbb M_{n}(\mathfrak A)$, we put, $$\|x\|_2:=(tr_n\otimes id)(x^*x)^{\frac{1}{2}}\in\mathfrak A^{+}.$$

\begin{theorem}\label{main}  Let $\tau$ be an $\mathfrak A$-trace on a $C^*$-algebra $A$. Consider the following assertions:

$(i)$ $\tau$ is amenable,

$(ii)$ there exists a sequence of c.c.p. module maps $\phi_n: A \to \mathbb M_{k(n)}(\mathfrak A)$ such that
$(tr_{k(n)}\otimes id)\circ\phi_n\to\tau$ in point-norm topology on $A$, as $n\to\infty$, and also $\|\phi_n(ab)-\phi_n(a)\phi_n(b)\|_2\to 0$ in $\mathfrak A$, as $n\to\infty$,
for all $a, b\in A$,

$(iii)$ the positive linear map $\mu_\tau$ on $A\odot A^{op}$ is min-continuous on $A \odot A^{op}$,

$(iv)$ the representation $\pi_\tau\otimes\pi_\tau^{op}: A\odot A^{op}\to \mathbb B(L^2(A,\tau))$ is min-continuous,

$(v)$ for any faithful representation $A\subseteq \mathbb B(Y)$, there exists a c.c.p. module map $\Phi:
\mathbb B(Y)\to\pi_\tau(A)^{''}$ extending $\pi_\tau$.

Then $(ii)\Rightarrow(iii)\Rightarrow(iv)$ and $(v)\Rightarrow(i)$. If moreover, $\mathbb B(L^2(A,\tau))$ is an injective object in the category of $\mathfrak A$-modules and c.p. module maps, then $(iv)\Rightarrow(v)$.
\end{theorem}
\begin{proof}
$(ii)\Rightarrow (iii)$. Let $(\phi_n)$ be as in $(ii)$,  consider the c.c.p. module maps $\phi_n^{op}: A^{op} \to \mathbb M_{k(n)}(\mathfrak A)^{op}$, and take the corresponding c.c.p. module map,
$$\phi_n\otimes\phi_n^{op}: A\otimes A^{op} \to \mathbb M_{k(n)}(\mathfrak A)\otimes\mathbb M_{k(n)}(\mathfrak A)^{op}=\mathbb B(L^2(\mathbb M_{k(n)}(\mathfrak A),tr_{k(n)}\otimes id)),$$
and compose it with 
$$\mu_n: x\in \mathbb B(L^2(\mathbb M_{k(n)}(\mathfrak A), tr_{k(n)}\otimes id))\mapsto \langle x\hat 1, \hat 1\rangle,$$
and observe that,
$$\mu_n\circ(\phi_n\otimes\phi_n^{op})(a\otimes b)=\langle\phi_n(a)J\phi_n^{op}(b^*)J\hat 1, \hat 1\rangle=(tr_{k(n)}\otimes id)\big(\phi_n(a)\phi_n(b)\big),$$
for $a,b\in A$.
Since $$0\leq |(tr_{k(n)}\otimes id)(x)|\leq\|x\|_2\quad \big(x\in \mathbb M_{k(n)}(\mathfrak A)\big),$$
as positive elements in $\mathfrak A$, we have
\begin{align*}
\|\mu_n(\phi_n\otimes\phi_n^{op}(a\otimes b))-\mu_\tau(a\otimes b)\|&=\|(tr_{k(n)}\otimes id)\big(\phi_n(a)\phi_n(b)\big)-\langle\pi_{\tau}(a)\pi_{\tau}(b)\hat 1, \hat 1\rangle\|\\
&\leq \|(tr_{k(n)}\otimes id)\big(\phi_n(a)\phi_n(b)-\phi_n(ab)\big)\|\\&-\|(tr_{k(n)}\otimes id)\big(\phi_n(ab)\big)-\tau(ab)\|\to 0,
\end{align*}
as $n\to\infty$,
for all $a, b\in A$. Since $\mu_\tau$ is point-norm limit of $min$-continuous maps, it is $min$-continuous.

$(iii)\Rightarrow (iv)$. By $(iii)$, $\mu_\tau$ extends to an $\mathfrak A$-retraction on $A\otimes A^{op}$, and so it has a KSGNS-representation,
$$\sigma: A\otimes A^{op}\to \mathbb B(L^2(A\otimes A^{op},\mu_\tau)).$$
By the universality of the KSGNS-representation \cite[Theorem 5.6]{lan}, there is a unitary $u\in \mathbb B\big(L^2(A\otimes A^{op},\mu_\tau), L^2(A,\tau)\big)$ such that $\pi_\tau\otimes\pi_\tau^{op}=u\sigma(\cdot)u^*$ on $A\odot A^{op}$. Thus the left hand side is $min$-continuous on $A\odot A^{op}$.

$(iv)\Rightarrow (v)$. Assume that $\mathbb B(L^2(A,\tau))$ is an injective object in the category of $\mathfrak A$-modules and c.p. module maps. By going to unitizations, we may assume that $A$ is unital. As in the classical case, one could use the Lance's trick: $A\otimes A^{op}\subseteq \mathbb B(Y)\otimes A^{op}$ and  $\pi_\tau \otimes \pi_\tau^{op}$ extends to a c.p. map $\Psi: \mathbb B(Y)\otimes A^{op}\to \mathbb B(L^2(A,\tau))$, having $A\otimes A^{op}$ in its multiplicative domain. Put, $\Phi(x)=\Psi(x\otimes 1)$, where $1$ is the unit of $A$. Then
$$\text{ran}(\Phi)\subseteq \Psi(\mathbb C1\otimes A^{op})=\pi_\tau^{op}(A)^{'}=\pi_\tau(A)^{''}\subseteq \pi_\tau(A)^{''}.$$

$(v)\Rightarrow (i)$. For any faithful representation $A\subseteq \mathbb B(Y)$, let $\Phi:
\mathbb B(Y)\to\pi_\tau(A)^{''}$ extend $\pi_\tau$ and put $\phi=\langle \Phi(\cdot)\hat 1,\hat 1)$. Since $A$ is in the multiplicative domain of $\Phi$, this is a c.c.p.  module map extending $\tau$, which is invariant under conjugation by unitaries in $A+\mathbb C I$, where $I$ is the identity operator on $Y$.
\end{proof}

We don't know if $(i)\Rightarrow (ii)$. The proof in the classical case uses approximation of state of $\mathbb B(H)$ (for a Hilbert space $H$) with normal states coming from finite rank positive elements, which has no counterpart in the vector valued case.

We recall that $A$ has $\mathfrak A$-WEP if for every faithful representation and module map $A\subseteq \mathbb B(H\otimes \mathfrak A)$ for a Hilbert space $H$, there is a u.c.p. admissible map $\varphi: \mathbb B(H\otimes A)\to A^{**}$ extending the identity on $A$ \cite{a2}. For definition of admissible maps, see \cite[Definition 2.1]{a2}.

\begin{proposition}
Assume that $\mathfrak A$ is a unital injective $C^*$-algebra such that for a minimal Stinespring dilation $(\pi,K,V)$  as in Lemma \ref{rep2}, $\mathbb B(K\otimes \mathfrak A)$ is a von Neumann algebra. If $A$ has $\mathfrak A$-WEP, then each  $\mathfrak A$-trace on $A$ is amenable.
\end{proposition}
\begin{proof}
We identify $A$ with its image under a  faithful representation in $\mathbb B(K\otimes \mathfrak A)$. Since $A$ has $\mathfrak A$-WEP, there is a u.c.p. module map $\Phi: \mathbb B(K\otimes \mathfrak A)\to A^{**}$, extending the identity map on $A$. For an $\mathfrak A$-trace $\tau:A\to \mathfrak A$, let $\psi$ be the restriction of $$(\tau^{**}\circ\Phi)\otimes {\rm id}: \mathbb B(K\otimes \mathfrak A)\bar\otimes  \mathfrak A^{**}\to \mathfrak A^{**}\bar\otimes \mathfrak A^{**}$$
to  $\mathbb B(K\otimes \mathfrak A)\bar\otimes \mathbb C1_{\mathfrak A^{**}}$, canonically identified with $\mathbb B(K\otimes \mathfrak A)$. Let $\mathbb E: \mathfrak A^{**}\bar\otimes \mathfrak A^{**}\to \mathfrak A$ be a conditional expectation and put $$\phi:=\mathbb E\circ\psi: \mathbb B(K\otimes \mathfrak A)\to \mathfrak A.$$
Since $A$ is in the multiplicative domain of $\Phi$, it is also in the multiplicative domain of $\psi$, hence $\phi$ is invariant under conjugation by unitaries of $A+\mathbb CI$. Finally, since $\mathbb E$ is identity on its range, $\phi$ extends $\tau$. The result now follows from Lemmas \ref{rep}, \ref{inj}.
\end{proof}

\section{quasidiagonality}

In this section we explore the module version of the notion of quasidiagonal (QD) $C^*$-algebras.

\begin{definition} \label{qd}
A $C^*$-module $A$ is called $\mathfrak A$-{\it quasidiagonal} (briefly, $\mathfrak A$-QD) if there exists a net of admissible c.c.p. maps $\phi_n: A \to \mathbb M_{k(n)}(\mathfrak A)$ which are approximately multiplicative and approximately isometric, i.e.,
$\|\phi_n(ab)-\phi_n(a)\phi_n(b)\|\to 0$ and $\|\phi_n(a)\|\to \|a\|,$ as $n\to\infty$,
for all $a, b\in A$; or equivalently, if for each finite set $\mathfrak F\subseteq A$ and $\varepsilon>0$, there is a positive integer $k$ and a c.c.p. module map $\phi: A \to \mathbb M_{k}(\mathfrak A)$ satisfying $\|\phi(ab)-\phi(a)\phi(b)\|<\varepsilon$ and $\|\phi(a)\|> \|a\|-\varepsilon,$
for all $a, b\in \mathfrak F$.
\end{definition}

When $A$ is unital and $\mathfrak A$ is a von Neumann algebra, an argument similar to that of \cite[7.1.4]{bo} shows that we may take the maps $\phi_n$ in the above definition to be u.c.p. module maps (if they exist).

Clearly $\mathbb M_{n}(\mathfrak A)$ is $\mathfrak A$-QD, for each positive integer $n$. As another immediate example, if $B=C_0(X)$ is a commutative $C^*$-algebra, then $A=B\otimes \mathfrak A=C_0(X, \mathfrak A)$ (with the right module action by multiplication) is $\mathfrak A$-QD (just take direct sums of point evaluations). More generally, we have the following notion.

\begin{definition} \label{rfd}
A $C^*$-module $A$ is called $\mathfrak A$-{\it residually finite dimensional} (briefly, $\mathfrak A$-RFD) if there is a net of $*$-homomorphisms and module maps $\pi_n: A \to \mathbb M_{k(n)}(\mathfrak A)$ such that $\oplus \pi_n: A\to \prod_{n} \mathbb M_{k(n)}(\mathfrak A)$ is faithful.
\end{definition}

Clearly each $\mathfrak A$-RFD $C^*$-module is also $\mathfrak A$-QD. Also if $B$ is an RFD $C^*$-algebra, then $A=B\otimes \mathfrak A$ is $\mathfrak A$-RFD (and in particular, this holds if $B$ is a Type I $C^*$-algebra). The property of being $\mathfrak A$-QD passes to direct products and subalgebras (which are also submodules) and so it also passes to direct sums. When $\mathfrak A$ is injective in the category of $\mathfrak A$-modules with c.c.p. module maps, then it also passes to direct limits with injective connecting maps (just as in  \cite[7.1.9]{bo}). Also, it behaves well wit respect to the minimal module tensor products, defined in \cite[Section 3]{a2}: if $A$ and $B$ are $\mathfrak A$-QD, so is $A\otimes_{\mathfrak A}^{min} B$ (c.f. \cite[7.1.12]{bo}).

We say that $A$ is  $\mathfrak A$-{\it stably finite} (briefly, $\mathfrak A$-SF) if $\mathbb M_{n}(\mathfrak A)\otimes_{\mathfrak A}^{min} A$ contains no proper isometry (i.e., an isometry $s$ with $ss^*\neq 1$), for each positive integer $n$. Similar to \cite[7.1.15]{bo}, if $A$ is $\mathfrak A$-QD, it is also $\mathfrak A$-SF.

\begin{definition} \label{qd2}
For a left Hilbert $\mathfrak A$-module $Y$, a subset $\Omega\subseteq \mathbb B(Y)$ is called {\it quasidiagonal} if for each finite sets $\mathfrak F\subseteq \Omega$ and $\mathfrak Y\subseteq Y$ and each $\varepsilon>0$ there is a finite rank projection $P\in\mathbb B(Y)$ such that $\|PT-TP\|<\varepsilon$ and $P=I$ on $\mathfrak Y$.
\end{definition}

Note that here, a finite rank operator is one whose range is a finitely generated submodule of $Y$. Also, unlike the classical case, since the submodules of $Y$ are not necessarily complemented, we may not assume that $\|Pv-v\|<\varepsilon$ for $v\in\mathfrak Y$ and prove the above stronger assumption by modifying $P$ with an orthogonal projection (c.f. the proof of \cite[7.2.3]{bo}). However, as in the above cited result, the stronger assumption gives the following property.

\begin{lemma} \label{proj} Let $Y$ be a countably generated Hilbert $\mathfrak A$-module and $\Omega\subseteq \mathbb B(Y)$ be norm separable and quasidiagonal. Then there is an increasing sequence $(P_n)$ of finite rank projections, converging strongly to the identity $I$ on $Y$, such that $\|[P_n,T]\|\to 0$, for $T\in\Omega$.
\end{lemma}

\begin{definition} \label{qd2}
A representation $\pi:A\to\mathbb B(Y)$ in a left Hilbert $\mathfrak A$-module $Y$ is called {\it quasidiagonal} if there is a sequence $(P_n)\subseteq \mathbb B(Y)$ of projections such that $P_n\pi(a)-\pi(a)P_n\in\mathbb K(Y)$ and $\|P_n\pi(a)-\pi(a)P_n\|\to 0$, as $n\to\infty$, for each $a\in A$. It is called {\it strongly quasidiagonal} if $\Omega:=\pi(A)$ is quasidiagonal set of adjointable operators.
\end{definition}

It is well known that the two notions are equivalent in the case of (separable) Hilbert space representations. We have the following weaker version of the classical result of Voiculescu \cite[7.2.5]{bo}.

\begin{theorem} [Voiculescu] \label{voi}
Assume that $\mathfrak A$ is unital and $A$ is unital and separable. The following are equivalent.

$(i)$ $A$ has a faithful strongly QD representation modulo the compacts in  $H\otimes\mathfrak A$, where $H$ is a separable Hilbert space,

$(ii)$ $A$ is $\mathfrak A$-QD with u.c.p. asymptotically multiplicative and isometric module maps.
\end{theorem}
\begin{proof}
$(i)\Rightarrow(ii)$. If $\pi:A\to \mathbb B(H\otimes\mathfrak A)$ is a faithful strongly QD representation, then by Lemma \ref{proj}, there is an increasing sequence $(P_n)$ of finite rank projections, say of rank $k(n)$, converging SOT to the identity $I$, such that, $\|[P_n,\pi(a)]\|\to 0$, for $a\in A$. Now the c.c.p. module maps $\phi_n: A\to P_n\mathbb B(H\otimes\mathfrak A)P_n\cong \mathbb M_{k(n)}(\mathfrak A)$ are asymptotically multiplicative and isometric.

$(ii)\Rightarrow(i)$. Let $\phi_n: A\to \mathbb M_{k(n)}(\mathfrak A)$ are u.c.p. asymptotically multiplicative and isometric module maps. Then
$$\Phi:=\oplus \phi_n: A\to \prod_{n=1}^{\infty}\mathbb M_{k(n)}(\mathfrak A)\subseteq\mathbb B\big(\oplus_{n=1}^{\infty}\ell^2_{k(n)}\otimes\mathfrak A\big)$$
a faithful  representation modulo the compacts, and for the canonical orthogonal projections $p_n: \ell^2\to \ell^2_{k(n)}$, $\Phi$ becomes strongly QD with the finite rank projections $P_n=p_n\otimes id$.
\end{proof}

Note that when $\mathfrak A$ is a von Neumann algebra, in $(ii)$ we could guarantee that the corresponding maps are also u.c.p. (not just c.c.p.). We don't know however if $(ii)$ implies that every faithful unital essential representation of $A$ in a countably generated $\mathfrak A$-correspondence $Y$ is QD. There are partial results in this direction by Kasparov: if $\mathfrak A$ is $\sigma$-unital and nuclear and $A$ is unital and separable,  faithfully represented in a Hilbert module of the form $H\otimes\mathfrak A$, for a separable Hilbert space $H$, then for each unital c.p. map $\pi: A/A\cap \mathbb K(H\otimes\mathfrak A)\to \mathbb B(H\otimes\mathfrak A)$, there is a sequence of isometries $(v_n)\subseteq \mathbb B(H\otimes\mathfrak A)$ with $\pi(a)-v_n^*av_n\in\mathbb K(H\otimes\mathfrak A)$ such that, $\|\pi(a)-v_n^*av_n\|\to 0$, as $n\to\infty$, for $a\in A$ \cite[Theorem 5]{k}. In particular, when $A\cap \mathbb K(H\otimes\mathfrak A)=0$ and $\pi$ is a unital representation, then for projections $P_n=v_nv_n^*$, $P_na-aP_n\in \mathbb K(H\otimes\mathfrak A)$ and $\|P_na-aP_n\|\to 0$, for each $a\in A$, that is, the embedding 
$$\iota: A\hookrightarrow \mathbb B(H\otimes\mathfrak A)$$
is $\mathfrak A$-QD. On the other hand, by \cite[Theorem 6]{k}, under the above conditions, $\iota$ is approximately equivalent to
$$\pi\oplus\iota: A\to \mathbb B(H\otimes\mathfrak A)$$
and so this representation is also  $\mathfrak A$-QD.

\begin{proposition}
Assume that $A$ is unital and $\mathfrak A$ is an injective von Neumann algebra.

$(i)$ If $A$ is $\mathfrak A$-QD with u.c.p. asymptotically multiplicative and isometric module maps, then $A$ has an amenable $\mathfrak A$-trace.

$(ii)$ If an $\mathfrak A$-retraction satisfies the condition $(ii)$ of Theorem \ref{main} with u.c.p. asymptotically multiplicative and isometric module maps, then it is amenable. 
\end{proposition}

\begin{proof}
We only show $(i)$, as part $(ii)$ is immediate from definition. If $\phi_n: A\to \mathbb M_{k(n)}(\mathfrak A)$ are u.c.p. asymptotically multiplicative and isometric module maps (in operator or Hilbert-Schmidt norm), and $A\subseteq \mathbb B(Y)$ be any faithful non degenerate representation in an $\mathfrak A$-correspondence $Y$ (with $A$ containing the identity of $\mathbb B(Y)$), then by the fact that $\mathfrak A$ is also injective in the category of $\mathfrak A$-modules \cite[Theorem 3.2]{fp}, we get c.c.p. module map extensions $\tilde\phi_n: \mathbb B(Y)\to \mathbb M_{k(n)}(\mathfrak A)$. The net consisting of the maps $(tr_{k(n)}\otimes {\rm id})\circ\tilde\phi_n: \mathbb B(Y)\to \mathfrak A$ has a cluster point in the point-ultraweak topology by \cite[1.3.7]{bo}. The restriction of this map is the given (or required) amenable trace.
\end{proof}

\section{local reflexivity}

Local reflexivity in the classical setting is closely related to exactness and several $min$-continuity properties studied by Kirchberg. In this section we develop the module versions of this notion and relate it to the result in \cite{a2}.

\begin{definition} \label{lr}
The $C^*$-module $A$ is called $\mathfrak A$-{\it locally reflexive} (briefly, $\mathfrak A$-LR) if for every operator subsystem and finitely generated submodule $E\subseteq A^{**}$, there exists a net of c.c.p.   maps $\phi_i: E \to A$  converging to ${\rm id}_E$ in the point-ultraweak topology.
\end{definition}

Note that we do not assume that the maps $\phi_i: E \to A$ are module map, but they would be approximately so, in the point-ultraweak topology.

Recall that an exact sequence $$0\rightarrow I\rightarrow A\xrightarrow{\pi} B\rightarrow 0$$
of $C^*$-modules, with arrows both $*$-homomorphisms and module maps, is {\it locally $\mathfrak A$-split} if for each finitely generated operator subspace and submodule $E\subseteq B$ there is a c.p. module map $\sigma: E\to A$ with $\pi\circ\sigma=$id$_E$ \cite{a2}. More generally, A c.c.p. module map $\phi: E\to A/J$ is called {\it $\mathfrak A$-liftable} if, for the quotient map $\pi: A\to A/J$, there is a c.c.p. module map $\sigma: E\to A$ with $\pi\circ\sigma=\phi$.

Let us first show that $\mathfrak A$-locally reflexivity passes to (and from) ideals and quotients (c.f. \cite[9.1.4]{bo}). But first we need the following module version of Arveson's Lemma \cite[Appendix C]{bo}.

\begin{lemma} [Arveson] \label{lift} Let $\mathfrak A$ be separable, $A$ be unital and
$I$ be a closed ideal and submodule. Then for each operator system and countably generated module $E$, the set of $\mathfrak A$-liftable module maps from $E$ to $A/J$ is closed in point-norm topology.
\end{lemma}
\begin{proof}
Let $\phi: E \to A/J$ be a c.c.p. module map and let $\psi^{'}_n : E \to A$ be c.c.p. module maps
with $\pi\circ\psi^{'}_n\to\phi$,
in the point-norm topology. Fix
countable dense subset $(\alpha_j)\subseteq \mathfrak A$ and countable generating set $(x_k)\subseteq E$. As in the proof of \cite[Lemma C2]{bo}, we may assume that,
$$\|\pi\circ\psi^{'}_n(x_k\cdot\alpha_j)-\phi(x_k\cdot\alpha_j)\|< 1/2^n\ \ (k,j<n)$$
and inductively find c.c.p. maps  $\psi_n : E \to A$ such that,
$$\|\pi\circ\psi_n(x_k\cdot\alpha_j)-\phi(x_k\cdot\alpha_j)\|< 1/2^n, \ \|\psi_{n+1}(x_k\cdot\alpha_j)-\psi_n(x_k\cdot\alpha_j)\|< 1/2^{n-1}\ \ (k,j<n).$$
Moreover, in the inductive step from $n$ to $n+1$,  since the c.c.p. map  $\psi_{n+1} : E \to A$ is defined via,
$$\psi_{n+1}= (1 - e_\lambda)^{\frac{1}{2}}\psi^{'}_{n+1}(1 - e_\lambda)^{\frac{1}{2}}+ e_\lambda^{\frac{1}{2}} \psi_n e_\lambda^{\frac{1}{2}},$$
for large enough index $\lambda$, in a quasicentral bounded approximate identity $(e_\lambda)$ of $J$ inside $A$, since $e_\lambda$ approximately commutes with the ranges of $\psi_n$ and $\psi^{'}_{n+1}$, and each $\psi^{'}_{n}$ is a module map, we may inductively guarantee that,
 $$\|\psi_{n}(x_k\cdot\alpha_j)-\psi_n(x_k)\cdot\alpha_j\|< 1/2^{n-1}\ \ (k,j<n).$$
Now the sequence $(\psi_n)$ of c.c.p.
maps converges in point-norm topology on a dense subset of $E$ (consisting of finite combinations of elements $x_k\cdot\alpha_j$ with coefficients in $\mathbb Q+i\mathbb Q$), and so everywhere, to a c.c.p. module map  $\psi: E \to A$ which lifts $\phi$.
\end{proof}

We omit the proof of the next lemma which adapts that of \cite[9.1.4]{bo}, using the above lemma.

\begin{lemma} \label{ext} Let $\mathfrak A$ be separable and $A$ be unital and
$$0\rightarrow I\rightarrow A\xrightarrow{\pi} B\rightarrow 0$$
be an exact sequence $C^*$-modules, with arrows both $*$-homomorphisms and module maps. Then $A$ is $\mathfrak A$-LR iff both $I$ and $B$ are $\mathfrak A$-LR and the extension is locally $\mathfrak A$-split.
\end{lemma}

It follows from  \cite[Proposition 3.14]{a2} that if $A$ is unital and $\mathfrak A$-LR and
$$0\to J\to A\to A/J\to 0$$
is an exact sequence with arrows both $*$-homomorphisms and module maps, then for each $C^*$-module $B$,
the sequence
$$0\to J\otimes_{\mathfrak A}^{min}B^{op}\to A\otimes_{\mathfrak A}^{min}B^{op}\to (A/J)\otimes_{\mathfrak A}^{min}B^{op}\to 0$$
is exact.

\begin{lemma} \label{lr2} Let $\mathfrak A$ be separable locally reflexive $C^*$-algebra, then $\mathfrak A$ is $\mathfrak A$-locally reflexive.
\end{lemma}
\begin{proof}
Let $(\alpha_j)$ be a countable dense subset of $\mathfrak A$. Given an operator subsystem and finitely generated submodule $E\subseteq \mathfrak A^{**}$ with generators $x_1,\cdots,x_k$, for each $n\geq 1$, let $E_{k,n}$ be the finite dimensional operator system generated by $1\in \mathfrak A^{**}$ and elements $x_i\cdot\alpha_j$ and their adjoints, for $1\leq i\leq k, 1\leq j\leq n$. Then there is a net $\psi_{m,k,n}: E_{k,n}\to \mathfrak A$ converging ultraweak to the identity on $E_{k,n}$, as $m\to \infty$. We may regards this as a  multi-index net, converging ultraweak to the identity on the dense subset $\cup_{k,n} E_{k,n}$ of $E$.
\end{proof}

\begin{corollary} \label{ext} Let $\mathfrak A$ be separable unital locally reflexive $C^*$-algebra, the a (non unital) $C^*$-module $A$ is  $\mathfrak A$-LR iff the unital $C^*$-module $A\oplus \mathfrak A$ is  $\mathfrak A$-LR.
\end{corollary}

If $A,B,$ and $C$ are $C^*$-modules and $\pi: A\otimes_{\mathfrak A}^{min}B^{op}\to C$ is a representation, then $\pi$ is induced by a $*$-homomorphism (still denoted by) $\pi: A\otimes_{min}B^{op}\to C$ vanishing on the $min$-closure of the ideal and submodule $I_{\mathfrak A}$ generated by elements of the form $a\cdot\alpha\otimes b-a\otimes b\cdot\alpha$, for $a\in A, b\in B$ and $\alpha\in\mathfrak A$. Hence $\pi=\pi_A\otimes\pi_B$ for representations $\pi_A: A\to C$ and $\pi_B: B\to C$ with commuting ranges, satisfying the compatibility condition, 
$$\pi_A(a\cdot\alpha)\pi_B(b)=\pi_A(a)\pi_B(b\cdot\alpha)\ \ (a\in A, b\in B, \alpha\in\mathfrak A).$$
For the canonical inclusion $\pi$ into $C:=(A\otimes_{\mathfrak A}^{min}B^{op})^{**}$, this gives a binormal map: $A^{**}\odot (B^{**})^{op}\to (A\otimes_{\mathfrak A}^{min}B^{op})^{**}$, vanishing on ideal and submodule $J_{\mathfrak A}$ generated by elements of the form $x\cdot\alpha\otimes y-x\otimes y\cdot\alpha$, for $x\in A^{**}, y\in B^{**}$ and $\alpha\in\mathfrak A$, which in turn induces a  binormal module map: $A^{**}\otimes_{\mathfrak A}^{min}(B^{**})^{op}\to (A\otimes_{\mathfrak A}^{min}B^{op})^{**}$. To observe that the latter map is also injective, take $x$ which is not in the $min$-closure of $J_{\mathfrak A}$ and use the fact that $A^*\odot B^*$ separates points and closed sets in $A^{**}\otimes_{min}(B^{**})^{op}$ (c.f., \cite[Exercise 3.1.5]{bo}) to find $\psi\in A^*\odot B^*$ which vanishes on $J_{\mathfrak A}$ and $\psi(x)=1$. Since $I_{\mathfrak A}\subseteq J_{\mathfrak A}$, it follows that $x\notin I_{\mathfrak A}^{\perp\perp}$, thus we have binormal module map inclusion  $$A^{**}\odot_{\mathfrak A}(B^{**})^{op}\hookrightarrow (A\otimes_{\mathfrak A}^{min}B^{op})^{**}.$$

\begin{definition} \label{prop}
The $C^*$-module $A$ is said to have property $C_\mathfrak A$, or  $C^{'}_\mathfrak A$, or $C^{''}_\mathfrak A$ if, respectively, the inclusion, 
 $$A^{**}\odot_{\mathfrak A}(B^{**})^{op}\hookrightarrow (A\otimes_{\mathfrak A}^{min}B^{op})^{**},$$
 or,
 $$A\odot_{\mathfrak A}(B^{**})^{op}\hookrightarrow (A\otimes_{\mathfrak A}^{min}B^{op})^{**},$$
 or,
 $$A^{**}\odot_{\mathfrak A} B^{op}\hookrightarrow (A\otimes_{\mathfrak A}^{min}B^{op})^{**},$$
 is $min$-continuous for any $C^*$-module $B$. 
 \end{definition}
 
It follows from \cite[Proposition 3.6]{a2} that any of the above properties passes to $C^*$-subalgebras which are also a submodule. Also, similar to \cite[9.2.4]{bo}, the first and third properties pass to quotients by closed ideals which are also a submodule. 

Let $E$ and $F$ be Banach spaces and, respectively, Banach right and left $\mathfrak A$-modules with compatible actions, and for $\mathfrak A$-correspondences $X$ and $Y$ and isometric inclusions $E\subseteq \mathbb B(X)$ and $F\subseteq \mathbb B(Y)^{op}$, give operator norms on $\mathbb M_n(E)\subseteq \mathbb B(X^n)$ and $\mathbb M_n(F)\subseteq \mathbb B(Y^n)^{op}$. For a linear module map $T: E\to F$ and amplification $T_n: \mathbb M_n(E)\to \mathbb M_n(F)$, put $\|T\|_{cb}=\sup_n\|T_n\|$. Also denote the completion of $E\odot_{\mathfrak A} F$ in $\mathbb B(X)\otimes_{\mathfrak A}^{min}\mathbb B(Y)^{op}$ by $E\otimes_{\mathfrak A}^{min} F$. Both this construction and the $cb$-norm are independent of the choice of embeddings.

For $E$ and $F$ as above, Let $B_{\mathfrak A}(E,\mathfrak A)$ be the Banach space of all bounded linear right module maps from $E$ to $\mathfrak A$. Similar to \cite[Theorem B.13]{bo}, we want to identify $B_{\mathfrak A}(E,\mathfrak A)\otimes_{\mathfrak A}^{min} F$ with $CB_{\mathfrak A}(E,F)$. In the next lemma we work with the case where the above isometric inclusions exists with countably generated $X$ and $Y$.

\begin{lemma} \label{iso} For $E$ and $F$ as above and $$z=\sum_{k=1}^{n}\phi_k\otimes y_k\in B_{\mathfrak A}(E,\mathfrak A)\odot_{\mathfrak A} F,$$ consider the module map $T_z:E\to F$ defined by $T_z(x)=\sum_{k=1}^{n} y_k\cdot\phi_k(x)$. Then $\|z\|_{min}=\|T_z\|_{cb}$ and the resulting isometric inclusion $$B_{\mathfrak A}(E,\mathfrak A)\otimes_{\mathfrak A}^{min} F\subseteq CB_{\mathfrak A}(E,F),$$ is surjective when $E$ or $F$ is finitely generated $\mathfrak A$-module.
\end{lemma}
\begin{proof}
Let $F\subseteq \mathbb B(Y)^{op}$, where $Y=H\otimes \mathfrak A$, with $H$  a separable Hilbert space. Then by \cite[Lemma 2.4]{a2}, there are finite rank projections $q_n\in \mathbb B(Y)$, say of rank $k(n)$, such that $q_n\uparrow 1$ (SOT) in $\mathbb B(Y)$. Let $\phi_n: \mathbb B(Y)\to\mathbb B(q_nY)=\mathbb M_{k(n)}(\mathfrak A)$ be the compression by $q_n$, and note that for $w_n=({\rm id}\otimes \phi_n)(z)$, $T_{w_n}=\phi_n\circ T_z$ in $CB_{\mathfrak A}(E,\mathbb M_{k(n)}(\mathfrak A))$. Thus
\begin{align*}
\|z\|_{min}&=\sup_n\|({\rm id}\otimes \phi_n)(z)\|_{B_{\mathfrak A}(E,\mathfrak A)\odot_{\mathfrak A}\mathbb M_{k(n)}(\mathfrak A)}\\
&=\sup_n\|T_{w_n}\|_{cb}\\
&=\sup_n\|\phi_n\circ T_z\|_{cb}\\
&=\|T_z\|_{cb}.
\end{align*}
The surjectivity in the finitely generated case is straightforward. 
\end{proof}

\begin{proposition} \label{c''}
Assume that $\mathfrak A$ is unital and separable, and for the $C^*$-module $A$, $A^{**}$ has a faithful representation in a countably generated $\mathfrak A$-correspondence. Then a $A$ is $\mathfrak A$-LR iff it has property $C_{\mathfrak A}^{''}$.
\end{proposition}
\begin{proof}
Let $A$ be $\mathfrak A$-LR and $B$ be any $C^*$-module. Let $z=\sum_{k=1}^{n} x_k\otimes b_k$ be an element in $A^{**}\odot B^{op}$ and $y=\sum_{j=1}^{m} y_j\cdot\alpha_j\otimes c_k-y_j\otimes c_k\cdot\alpha_j$ be a typical element in the ideal $J_{\mathfrak A}$ of $A^{**}\odot B^{op}$. Let $E$ be the operator system and finitely generated sumbodule of $A^{**}$ generated by the first legs of all the elementary tensors appeared in the decompositions of $z$ and $y$ above. By assumption, there is a net of c.c.p. maps $\phi_i: E\to A$, converging to the identity of $E$ in point-ultraweak topology. We then have, $(\phi_i\otimes {\rm id}_B)(z+y)\to z+y$, in the ultraweak topology of $(A\otimes_{min} B^{op})^{**}$. Thus,
$$\|z+y\|_{(A\otimes_{min} B^{op})^{**}}\leq\liminf_i\|(\phi_i\otimes {\rm id}_B)(z+y)\|_{A\otimes_{min} B^{op}}\leq \|z+y\|_{A^{**}\otimes_{min} B^{op}}.$$
Let $I_{\mathfrak A}$ be the corresponding ideal of $A\odot B^{op}$, then since $A$ is weak${}^*$ dense in $A^{**}$, $y\in I_{\mathfrak A}^{\perp\perp}$, that is, $\bar y=0$ in $(A\otimes_{\mathfrak A}^{min} B^{op})^{**}$. Therefore, taking infimum over all $y$'s, we get $\|\bar z\|_{(A\otimes_{\mathfrak A}^{min} B^{op})^{**}}\leq \|\bar z\|_{A^{**}\otimes_{\mathfrak A}^{min} B^{op}}$, where $\bar z$ in the left and right hand sides are cosets of $z$ in the corresponding space. Therefore, $A$ has property $C_{\mathfrak A}^{''}$.

Conversely, if $A$ has property $C_{\mathfrak A}^{''}$, let $E$ be the operator system and finitely generated sumbodule of $A^{**}$ with generators $x_1,\cdots,x_k$ and $(\alpha_j)$ be a countable dense subset of $\mathfrak A$. Similar to the proof of Lemma \ref{lr2}, we could construct an increasing double-indexed sequence of finitely generated operator systems $E_{j,n}\subseteq A^{**}$, whose union is dense in $E$. By the above lemma, the inclusion $E\hookrightarrow A^{**}$ corresponds to an element $z\in B_{\mathfrak A}(E,\mathfrak A)\otimes_{\mathfrak A}^{min} A^{**}$ with $\|z\|_{min}=1.$ Let $B_{\mathfrak A}(E,\mathfrak A)\subseteq \mathbb B(X)$ isometrically, for a $\mathfrak A$-correspondence $X$. Then by  property $C_{\mathfrak A}^{''}$, we have isometric inclusion
$$\mathbb B(X)\otimes_{\mathfrak A}^{min} A^{**}\hookrightarrow (\mathbb B(X)\otimes_{\mathfrak A}^{min} A)^{**}.$$ 
Next, as in the proof of \cite[Proposition 3.6]{a2}, we get isometric inclusions,
$$B_{\mathfrak A}(E,\mathfrak A)\otimes_{\mathfrak A}^{min} A^{**}\hookrightarrow \mathbb B(X)\otimes_{\mathfrak A}^{min} A^{**},$$
and, 
$$(B_{\mathfrak A}(E,\mathfrak A)\otimes_{\mathfrak A}^{min} A)^{**}\hookrightarrow (\mathbb B(X)\otimes_{\mathfrak A}^{min} A)^{**}.$$
Therefore, the map, 
$$B_{\mathfrak A}(E,\mathfrak A)\otimes_{\mathfrak A}^{min} A^{**}\rightarrow (B_{\mathfrak A}(E,\mathfrak A)\otimes_{\mathfrak A}^{min} A)^{**},$$
is isometric. Thus, $\|z\|_{(B_{\mathfrak A}(E,\mathfrak A)\otimes_{\mathfrak A}^{min} A)^{**}}=1.$ Choose a net $(z_i)$ in $B_{\mathfrak A}(E,\mathfrak A)\otimes_{\mathfrak A}^{min} A$ converging weak${}^*$ to $z$ such that $\|z_i\|\leq 1$, for each $i$, and apply the lemma again to get a net of c.c. module maps $\phi_i: E\to A$ converging to ${\rm id}_E$ in the point-ultraweak topology. The restriction $\phi_{i,j,n}$ of $\phi_i$ to $E_{j,n}$ is a c.c. map, and as in the proof of \cite[9.2.5]{bo}, could be replaced by a net of c.c.p. maps (not  module maps anymore) converging (as a multi-index net) to the identity on the dense subset $\cup_{j,n} E_{j,n}$, in point-ultraweak topology. Therefore, $A$ is $\mathfrak A$-LR.
\end{proof}

The next proposition is proved as in \cite[9.2.7]{bo}, in which instead of \cite[3.7.6]{bo}, we use \cite[Proposition 3.14(ii)]{a2}.

\begin{proposition} \label{c'}
A $C^*$-module $A$ with property $C_{\mathfrak A}^{'}$ is $\mathfrak A$-exact. 
\end{proposition}

We don't know if the converse is also true. Also not having the analog of Dadarlat's embedding theorem \cite[8.2.4]{bo},  we don't know if every $\mathfrak A$-exact $C^*$-module is a subquotient of an $\mathfrak A$-nuclear $C^*$-module.

The next lemma is proved as in \cite[9.3.2]{bo}, in which instead of \cite[3.8.5]{bo}, we use \cite[Theorem 3.17]{a2}.

\begin{proposition} \label{c} 
A $C^*$-module $A$ has property $C_{\mathfrak A}$ if $A^{**}$ is $\mathfrak A$-semidiscrete.  
\end{proposition}

Again, we don't know if an $\mathfrak A$-injective $C^*$-module is $\mathfrak A$-semidiscrete (or $\mathfrak A^{**}$-semidiscrete), so we could not check if for an $\mathfrak A$-nuclear $C^*$-module $A$, $A^{**}$ is $\mathfrak A$-semidiscrete (or $\mathfrak A^{**}$-semidiscrete). In particular, we don't know if $\mathfrak A$-exact (or even $\mathfrak A$-nuclear) $C^*$-modules have  property $C_{\mathfrak A}$. This also closes the (most obvious) path to show that $\mathfrak A$-exact $C^*$-modules are $\mathfrak A$-locally reflexive; or that, the quotients of $\mathfrak A$-exact (resp., $\mathfrak A$-nuclear) $C^*$-modules by closed ideals, which are also submodules, are again $\mathfrak A$-exact (resp., $\mathfrak A$-nuclear).

\bibliographystyle{line}
\bibliography{JAMS-paper}

\end{document}